\documentclass[12pt, letterpaper]{amsart}

\newcommand{\supp}{\operatorname{supp}}

\usepackage{amssymb}
\usepackage{amsthm}
\usepackage{amsxtra}
\usepackage{graphicx}
\usepackage{color}
\usepackage{xcolor}
\usepackage{mathrsfs}

\newtheorem{theorem}{Theorem}

\newtheorem{lemma}[theorem]{Lemma}

\theoremstyle{remark}
\newtheorem{remark}[theorem]{Remark}

\numberwithin{equation}{section}

\numberwithin{theorem}{section}

\numberwithin{table}{section}

\numberwithin{figure}{section}

\ifx\pdfoutput\undefined
  \DeclareGraphicsExtensions{.pstex, .eps}
\else
  \ifx\pdfoutput\relax
    \DeclareGraphicsExtensions{.pstex, .eps}
  \else
    \ifnum\pdfoutput>0
      \DeclareGraphicsExtensions{.pdf}
    \else
      \DeclareGraphicsExtensions{.pstex, .eps}
    \fi
  \fi
\fi

\title[Solutions of the 2D radially symmetric VM with an initial focusing phase]{Solutions of the 2D radially symmetric Vlasov-Maxwell system with an initial focusing phase}

\author{Katherine Zhiyuan Zhang}
\address{Courant Institute of Mathematical Sciences, New York University}

\begin{document}

\maketitle

\begin{abstract}

We study radially symmetric solutions to the 2D Vlasov-Maxwell system and construct solutions that initially possess arbitrarily small $C^k$ norms ($k \geq 1$) for the charge densities and the electric fields, but attain arbitrarily large $L^\infty$ norms of them at some later time.

\end{abstract}


\section{Introduction}

The behavior of solutions for Vlasov models describing plasma has been an important topic that caught wide attention. An important example is the Vlasov-Poisson equation (VP) (in which the magnetic field is absent) 
\begin{equation}
\partial_t f + v \cdot \nabla_x f +  E  \cdot \nabla_v f =0 ,
\end{equation}
\begin{equation}
E =  \nabla_x \phi   ,  \quad  \Delta_x \phi   = \int_{\mathbb{R}^3} f dv . 
\end{equation}


In Ben-Artzi-Calogero-Pankavich \cite{BCP2}, \cite{BCP1} and Zhang \cite{Z1}, it is shown that there exist spherically symmetric solutions of 3D VP such that the particle density and the electric field are initially as small as desired, but become as large as desired at some later time (called "focusing solutions"). Namely, for any positive constants $\eta$, $N$, $b$, $\epsilon_0$ and integer $\beta \geq 1$, there exists a smooth, spherically symmetric solution of VP, such that 
\begin{equation}
\supp_x f(0, x, v) \subset \{  \frac{1}{2} b \leq |x| \leq \frac{3}{2} b   \} 
\end{equation}
and 
\begin{equation}  
\| \rho (0) \|_{C^\beta}  , \ \| E (0) \|_{C^\beta} \leq \eta  , 
\end{equation}
while for some $T >0$, 
\begin{equation} 
\| \rho (T) \|_{L^\infty_{|x| \leq \epsilon_0}}  , \ \| E (T) \|_{L^\infty} \geq N  . 
\end{equation}
These results were inspired by a similar study by G. Rein and L. Teagert \cite{RT1}, which proves the existence of focusing solutions for the gravitational Vlasov-Poisson system, in which the electric force provides an attractive effect instead of a repulsive effect on the particles. It is striking that in contrast to the results in \cite{BCP2}, \cite{BCP1}, \cite{Z1}, a recent result by S. Pankavich \cite{Pankavich1} in 2021 shows that any spherically symmetric solution (with initial data in $C^1 (\mathbb{R}^3 \times \mathbb{R}^3)$) must decay for $t$ sufficiently large. Namely, there exists $C>0$, such that $ \| \rho (t) \|_{L^\infty} \leq C (1+t)^{-3} \ , \ \| E(t) \|_{L^\infty} \leq C (1+t)^{-2}  $ for all $t \geq 0$. There is no contradiction between the conclusions in \cite{BCP2}, \cite{BCP1}, \cite{Z1} and \cite{Pankavich1}. 

However, for a system that involves a magnetic field, the study of focusing solutions is absent. In this paper, we consider the Vlasov-Maxwell system (VM) on the 2D plane: 
\begin{equation} \label{E: VM-V}
\partial_t f + v \cdot \nabla_x f +  (E_1+ v_2 B, E_2 - v_1 B)   \cdot \nabla_v f =0  ,
\end{equation}
\begin{equation}  \label{E: VM-M}
\begin{split}
& \partial_{x_1} E_1  + \partial_{x_2} E_2 = \rho , \\
& \partial_t E_1 = \partial_{x_2} B - j_1 , \\
& \partial_t E_2 = - \partial_{x_1} B - j_2 , \\
& \partial_t B = \partial_{x_2} E_1 - \partial_{x_1} E_2 . \\
\end{split}
\end{equation}
Here, $f (t, x, v) \geq 0$ is the density distribution of the particles. In the equation, $x \in \mathbb{R}^2$ is the particle position, $v \in \mathbb{R}^2$ is the particle momentum, and $E= (E_1, E_2)$ is the electric field, $B$ is the magnetic field. The macroscopic charge density $\rho := \int_{\mathbb{R}^2} f dv $, and the $i$-th component of the current density is $j_i :=  \int_{\mathbb{R}^2} v_i f dv $ for $j=1$, $2$. The system VM enjoys the conservation of the total mass
\begin{equation}
M (t) = \int_{\mathbb{R}^2} \int_{\mathbb{R}^2}  f (t, x, v) dv dx =  \int_{\mathbb{R}^2} \int_{\mathbb{R}^2}  f_0 ( x, v) dv dx : = M  , \  \forall t   .
\end{equation}
In this paper, we consider the VM equation \eqref{E: VM-V} -- \eqref{E: VM-M} with initial data supported on an arbitrary annulus that has small $L^\infty$ norm, and obtain radially symmetric solutions that become large and are concentrated near the origin at some later time. Specifically, we prove the following result:

\begin{theorem} \label{mainresultVM}
For any integer $\beta \geq 1$ and any positive constants $\eta$, $N$, $\epsilon_0$, there exists a smooth, radially symmetric solution of the VM equation \eqref{E: VM-V} -- \eqref{E: VM-M}, such that  
\begin{equation}
\supp_x f(0, x, v) \subset \{  \frac{1}{2}  \leq |x| \leq 1   \} 
\end{equation}
and 
\begin{equation}  
\| \rho (0) \|_{C^\beta} \leq \eta  , \ \| E (0) \|_{C^\beta} \leq \eta   , \ \| B (0) \|_{C^\beta} \leq \eta  , 
\end{equation}
while for some $T >0$, 
\begin{equation}
\| \rho (T) \|_{L^\infty_{|x| \leq \epsilon_0}}  \geq N , \quad  \| E_r (T) \|_{L^\infty_{|x| \leq \epsilon_0}} \geq N 
\end{equation}
provided that a $C^1$ solution $(f, E, B)$ exists until the time $T$. 
Here $E_r$ is the $r$-component of the electric field $E$ under the polar coordinates.
\end{theorem}

\begin{remark}
The solution constructed in Theorem \ref{mainresultVM} has nontrivial magnetic field. On the other hand, it is more tricky to see if $\| E_\varphi \|_{L^\infty}$ or $\| B \|_{L^\infty}$ grows large in $[0, T]$ ($E_\varphi$ is the $\varphi$-component of the electric field $E$ under the polar coordinates), since the equations of $E_\varphi$ and $B$ are intertwined, see Lemma \ref{L:rep-E-B}. 
\end{remark}

\begin{remark}
The assumption $\supp_x f(0, x, v) \subset \{  \frac{1}{2}  \leq |x| \leq 1 \} $ can be replaced by $\supp_x f(0, x, v) \subset \{  \frac{1}{2} b \leq |x| \leq \frac{3}{2} b  \} $ with a minor change in the proof.
\end{remark}

\begin{remark}
Theorem \ref{mainresultVM} is \emph{conditional} since there is no known global existence result for classical solutions of \eqref{E: VM-V} -- \eqref{E: VM-M} which contains a nonzero magnetic field. In fact, the existence of global classical solutions to the non-relativistic VM system remains an unresolved challenging problem in the field.
\end{remark}



The Vlasov-Poisson and the Vlasov-Maxwell systems are among the most important models for dilute plasmas (see \cite{Friedberg1}, \cite{Nicholson1}). The issue of the existence of global-in-time solutions for Vlasov-Poisson is addressed in \cite{BD1}, \cite{GS1}, \cite{GS4}, \cite{LP1}, \cite{P1}, etc.. For the Vlasov-Maxwell equation, this question is much harder and more intriguing. Various progress has been made throughout the last several decades, see, for example, \cite{GS2}, \cite{GS3}, \cite{GStrauss1}, \cite{KS1}, \cite{LS1}, \cite{W1}. The reader can also refer to Glassey \cite{G1} for a thorough introduction of the basic Cauchy theory of the Vlasov-Poisson and the Vlasov-Maxwell systems. Moreover, for Vlasov-Poisson, a series of study on its detailed large time behavior has been carried out, see, for instance, \cite{GPS1}, \cite{GPS2}, \cite{GPS3}, \cite{GPS4}, \cite{H1}, \cite{IR1}, \cite{IPWW1}, \cite{Pankavich1}. 

The proof of our main result Theorem \ref{mainresultVM} relies upon an analysis of the particle trajectory and the mass conservation property of the equation. With a magnetic field, the analysis of the particle trajectory is more complicated. We make use of the dimension reduction brought in by the radial symmetry setting, and carry out an estimate on the electromagnetic field, so as to obtain enough information of the particle trajectory during the time interval during which the focusing behavior happens.

The contents of the paper are arranged as follows. In Section \ref{S:Settings}, we describe the basic settings on the problem, in particular we discuss some representation formulas for the electromagnetic field. In Section \ref{S:Particle-trajectories}, we give key lemmas that provide estimates for the electromagnetic field and description of the particle trajectories, which allow us to observe the focusing phenomena. Section \ref{S:mainthmproof} is devoted to the proof of Theorem \ref{mainresultVM}, which involves a careful selection of parameters and computation of the norms of $\rho (t, x)$ and $E_r(t, x)$.

\section{Settings}
\label{S:Settings}

We use the polar coordinates $(r, \varphi )$. We assume radial symmetry in the initial data, so $f(0) =f(0, r)$, $(E, B)(0) =(E, B) (0, r)$. We write \eqref{E: VM-V} -- \eqref{E: VM-M} in the polar coordinates:
\begin{equation} \label{E: VM-V-p}
\partial_t f + v_r \partial_r f + v_\varphi \frac{1}{r} \partial_\varphi f +  (E_r+ v_\varphi B + \frac{v_\varphi^2}{r}) \partial_{v_r} f  + ( E_\varphi - v_r B - \frac{v_r v_\varphi}{r})  \partial_{v_\varphi} f =0  ,
\end{equation}
\begin{equation}  \label{E: VM-M-p}
\begin{split}
& \frac{1}{r} \partial_{r} (r E_r)  +  \frac{1}{r}  \partial_{\varphi}  E_\varphi = \rho , \\
& \partial_t E_r = \frac{1}{r} \partial_{\varphi} B - j_r , \\
& \partial_t E_\varphi = - \partial_r B - j_\varphi , \\
& \partial_t B =  \frac{1}{r} \partial_{\varphi} E_r - \frac{1}{r} \partial_{r} (r E_\varphi ) , \\
\end{split}
\end{equation}
and assign the following radially symmetric initial and boundary conditions:
\begin{equation}  \label{E: VM-M-p-ibc}
\begin{split}
& f(0, r, v_r, v_\varphi) = f_0 (r, v_r, v_\varphi) \geq 0 , \\
& E_\varphi (0, r) = E_{\varphi, 0} (r) , \\
& B (0, r) = B_0 (r) , \\
& E_r (0, 0) = 0 . \\
\end{split}
\end{equation}
where $E_{\varphi, 0}$ and $B_0$ are given $C^1$ functions satisfying
\begin{equation}
\| (E_{\varphi, 0}, B_0 ) \|_{L^\infty} \leq \eta . 
\end{equation}

Writing the equations 
\begin{equation}
\dot{X}  = V,  \  \dot{V}  =  E (t, X ) + B  (V_2, -V_1)   
\end{equation}
under the polar coordinates, we obtain that the forward characteristics $(r(s), \varphi(s) )$ of the Vlasov equation are described by the following ODE system:
\begin{equation} \label{particletrajectoryODEVM}
\begin{split}
& \frac{d^2 r}{ds^2} = r \big(\frac{d \varphi}{ds} \big)^2  + E_r + r\dot{\varphi} B   , \\
& \frac{d}{ds} \big( r^2 \frac{d \varphi}{ds} \big)  = r E_\varphi - r \dot{r} B  .  \\
\end{split}
\end{equation}
for $s \geq 0$, with the initial conditions
\begin{equation}  \label{particletrajectoryinitialcondition-VM}
r(0) =r_0  , \ \varphi(0) =\varphi_0 , \  \dot{r} (0) =\dot{r}_0 , \  \dot{\varphi} (0) =\dot{\varphi}_0   . 
\end{equation}



We denote
\begin{equation}
S(t) := \{  (r,\varphi,  \dot{r} , \dot{\varphi} ) : f(t, r,\varphi,   \dot{r} , \dot{\varphi}  ) > 0  \}  . 
\end{equation}
In particular, 
\begin{equation}
S(0)  := \{  (r,\varphi,  \dot{r} , \dot{\varphi} ) : f_0(r,\varphi,  \dot{r} , \dot{\varphi}   ) > 0  \} . 
\end{equation}
(Notice that for the radial symmetry of the solution to \eqref{E: VM-V} -- \eqref{E: VM-M} can be propagated, see \cite{JSW1}, so we have $ f_0(r,\varphi,  \dot{r} , \dot{\varphi}   ) =  f_0(r, \dot{r} , \dot{\varphi}   )$, $f(t, r,\varphi,   \dot{r} , \dot{\varphi}  ) = f(t, r, \dot{r} , \dot{\varphi}  )$.) 

We introduce the following

\begin{lemma}  \label{L:rep-E-B}
Let $(f, E_r, E_\varphi, B)$ be a solution to the 2D radially symmetric VM. Let $P_\pm = r E_\varphi \pm r B$, $t_1 (t, r) = \max \{0, t-r \}$. (Notice that $t_1 (t, r) = \max \{0, t-r \} = 0$ when $t < r$.) We have
\begin{equation}   \label{E:rep-E-B-eq4}
r E_r (t, r) =   \int_0^r s \rho (t, s) ds , 
\end{equation}
\begin{equation}   \label{E:rep-E-B-eq5}
\begin{split}
& r E_\varphi (t, r) 
= \frac{1}{2} [ P_+ (t_1(t, r), r-t + t_1 (t, r) ) + P_- (0, r+t   )  ] \\
& \quad + \frac{1}{2} \int_{t_1(t, r)}^t (B- r j_\varphi) (\tau, r - t + \tau) d \tau +  \frac{1}{2} \int_{0}^t (B- r j_\varphi) (\tau, r + t - \tau) d \tau  ,   \\
\end{split}
\end{equation}
\begin{equation}   \label{E:rep-E-B-eq6}
\begin{split}
& r B (t, r) 
= \frac{1}{2} [ P_+ (t_1(t, r), r-t + t_1 (t, r) ) - P_- (0, r+t  )  ]  \\
& \quad + \frac{1}{2} \int_{t_1(t, r)}^t (B- r j_\varphi) (\tau, r - t + \tau) d \tau -  \frac{1}{2} \int_{0}^t (B- r j_\varphi) (\tau, r + t - \tau) d \tau  .   \\
\end{split}
\end{equation}
\end{lemma}

\begin{proof}
The proof is similar to the argument in Section 2.2 in \cite{JSW1} so we omit it. 
\end{proof}

\section{Estimate on the Field and the Particle Trajectories}
\label{S:Particle-trajectories}

We introduce the following lemma estimating the electromagnetic field.



\begin{lemma}  \label{L:est-E-B}
Let $(f, E_r, E_\varphi, B)$ be a solution to the 2D radially symmetric VM. Assume 
\begin{equation}  \label{E:est-E-B-eq1}
\begin{split}
M
& = \iint f (t, x, v) dv dx = \iint f_0 (x, v) dv dx  \\
& = 2 \pi \int_0^{+\infty} \int_{\mathbb{R}^2} r f_0 (r, v) dv dr = 2 \pi \int_0^{+\infty}  r \rho_0 (r) dr  < +\infty .  
\end{split}
\end{equation}
Then we have, for any $T_1 >0$,
\begin{equation}  \label{E:est-E-B-eq2}
| E_r (T_1, R) |  \leq \frac{M}{2 \pi R} . 
\end{equation}
Moreover, let $(r, \varphi) (t,  r_0 , \varphi_0 ,  \dot{r}_0 , \dot{\varphi}_0 )$ solve the ODE \eqref{particletrajectoryODEVM} -- \eqref{particletrajectoryinitialcondition-VM}. 
If for all $t \in [0, T_1]$,
\begin{equation}   \label{E:est-E-B-eq3}
r (t,  r_0 , \varphi_0 ,  \dot{r}_0 , \dot{\varphi}_0 )  \leq r_0 , \quad  \inf_{ (r_0 ,\varphi_0,  \dot{r}_0 , \dot{\varphi}_0 )\in S(0)} r (t,  r_0 , \varphi_0 ,  \dot{r}_0 , \dot{\varphi}_0 )  \geq 6 t ,  \quad   r(t)^2 \dot{\varphi} (t)  \leq \frac{4}{3} r_0^2 \dot{\varphi}_0 ,  
\end{equation}
and for some positive numbers $K$, $K_1$,
\begin{equation}   \label{E:est-E-B-eq7}
\begin{split}
& \rho (t, r) \leq K , \\
& \sup_{ (r_0 ,\varphi_0,  \dot{r}_0 , \dot{\varphi}_0 )\in S(0)}  |\dot{\varphi}_0|  \leq  K_1  ,  \\ 
& \sup_{ (r_0 ,\varphi_0,  \dot{r}_0 , \dot{\varphi}_0 )\in S(0)}  |r_0|  \leq  1  ,  \\ 
\end{split}
\end{equation}
with $\| ( E_{\varphi, 0}, B_0) \|_{L^\infty} \leq \frac{2KK_1T_1}{R}$, 
then for any $R \geq 6T_1$, 
\begin{equation}  \label{E:est-E-B-eq4}
| ( E_\varphi, B) (T_1, R) | \leq \| ( E_{\varphi, 0}, B_0) \|_{L^\infty} + \frac{6KK_1T_1}{R} .
\end{equation}
\end{lemma}

\begin{proof}
\eqref{E:est-E-B-eq2} follows from \eqref{E:rep-E-B-eq4} -- \eqref{E:rep-E-B-eq6}, due to that $| \int_0^r s \rho (T_1, s) ds | \leq \frac{M}{2 \pi }$ for all $T_1> 0$. 



For the proof of \eqref{E:est-E-B-eq4}, we first observe
\begin{equation}
\int_0^{T_1}  ( R j_\varphi) (\tau, R + T_1 - \tau)  d \tau  \leq 2 KK_1 T_1  . 
\end{equation}
This is because on the support of the integral we have $R j_\varphi (\tau, R+T_1-\tau) = \int R (R+T_1-\tau) \dot{\varphi} f (\tau, R+T_1-\tau, v) dv \leq \int \frac{3}{2} R^2 \dot{\varphi} f (\tau, R+T_1-\tau, v) dv \leq \int 2 r_0^2 \dot{\varphi}_0 f (\tau, R+T_1-\tau, v)  dv \leq 2 K_1  \int f (\tau, R+T_1-\tau, v)  dv  =  2K_1  \rho (\tau, R+T_1-\tau, v)  $. Here we have used $ R^2 \dot{\varphi}  \leq \frac{4}{3} r_0^2 \dot{\varphi}_0 $, which follows by the assumption of the lemma.

%


Assume that for some $t \in [0, T_1]$, $\sup_{r >0} |rB(t, r)| \geq r \| ( E_{\varphi, 0}, B_0) \|_{L^\infty}  + 6KK_1T_1 $. Let $t_0 := \inf \{ t>0 : \sup_{r >0} |rB(t, r)| \geq  r \| ( E_{\varphi, 0}, B_0) \|_{L^\infty}  + 6KK_1 T_1 \}$. We have, for some $r>0$,
\begin{equation}  \label{E:est-E-B-eq5}
\begin{split}
& r \| ( E_{\varphi, 0}, B_0) \|_{L^\infty}  +   6KK_1 T_1 \leq r B(t_0, r)  \\
& \leq r \| ( E_{\varphi, 0}, B_0) \|_{L^\infty}  +  \frac{1}{2} \int_0^{t_0} B (\tau, r-t_0+\tau) d\tau - \frac{1}{2} \int_0^{t_0} B (\tau, r+t_0-\tau) d\tau + 2KK_1T_1 ,  \\
\end{split}
\end{equation}
which implies
\begin{equation}
\frac{1}{2} \int_0^{t_0} B (\tau, r-t_0+\tau) d\tau - \frac{1}{2} \int_0^{t_0} B (\tau, r+t_0-\tau) d\tau \geq 4KK_1 T_1 . 
\end{equation}
Therefore there must exists some $t' \in [0, t_0]$, $r' \in [r-t_0, r+t_0]$, such that $B(t', r') \geq \frac{4KK_1 T_1}{t_0} \geq \frac{20KK_1 T_1}{r-t_0} \geq \frac{20KK_1 T_1}{r}$ (here we used that $t_0 < \frac{r-t_0}{5}$, which follows from \eqref{E:est-E-B-eq3}). Since $r' \in [r-t_0, r+t_0]$, we have $r' \geq \frac{r}{2}$. Hence $B(t', r') \geq \frac{20KK_1 T_1}{r} \geq \frac{10KK_1 T_1}{r'}  >  \| ( E_{\varphi, 0}, B_0) \|_{L^\infty} +  \frac{6KK_1 T_1}{r'}$. A contradiction to the definition of $t_0$. Using \eqref{E:rep-E-B-eq5}, we obtain the estimate for $E_\varphi$. 

This completes the proof of the lemma.
\end{proof}

We also give the following lemma, which describes the behavior of the particle trajectories of VM when the particle trajectories satisfy $0 < r_0 \leq 1$, $\dot{r}_0<0$, $\dot{\varphi}_0 > 0$ at time $0$. 

\begin{lemma} \label{L:behaviorofcharacteristicsVPM}
Let $0 < r_0 \leq 1$, $\dot{r}_0<0$, $\dot{\varphi}_0 > 0$, and let $(r(t), \varphi(t))$ be a solution to \eqref{particletrajectoryODEVM} and \eqref{particletrajectoryinitialcondition-VM} for all $t \geq 0$. Assume that $ |(E_r, E_\varphi, B) (t)| \leq \frac{m}{3r} $ for some constant $m>0$ (independent of $r$, $t$) on the support of $f(t)$. 
Define:
$$ \mathcal{A} : = \dot{r}_0^2 + r_0^{-2} (  \frac{1}{2} m r_0 + r_0^2 \dot{\varphi}_0 )^2  + 2 m r_0^{-1} (   \frac{1}{2} m r_0+ r_0^2 \dot{\varphi}_0 ) -2 m \ln r_0  , $$ 
$$ \mathcal{B} :=  2 m r_0 (   \frac{1}{2} m r_0+ r_0^2 \dot{\varphi}_0 ) + (  \frac{1}{2} m r_0 + r_0^2 \dot{\varphi}_0 )^2 +2m . $$
Assume that
\begin{equation}
\begin{split}
& - \frac{1}{2}  mr_0 + r_0^2  \dot{\varphi}_0 > 0 ,  \\
& \mathcal{A} > 0 , \quad  \mathcal{B} > 0 , \quad  \mathcal{A} r_0^2 - \mathcal{B}  > 0 . \\
\end{split}
\end{equation}
Also, we assume that $T_0 \leq \frac{1}{100} r_0$, and for $s \in [0, T_0]$, $\dot{r} (s) < 0$. Then for all $s \in [0, T_0]$, 
$$ r(s)^2    \leq (r_0 - \sqrt{\mathcal{A}  - \mathcal{B} r_0^{-2}} s )^2 + \mathcal{B} r_0^{-2} s^2 .   $$
\end{lemma}

\begin{proof}
Integrating the second line of \eqref{particletrajectoryODEVM} over $[0, s]$ and using $ |(E_r, E_\varphi, B) (t)| \leq \frac{m}{3r} $, we learn that 
$$ r^{-2} ( -ms+ \frac{1}{3} mr - \frac{1}{3} m r_0+ r_0^2 \dot{\varphi}_0 )   \leq \dot{\varphi} \leq r^{-2} ( ms -\frac{1}{3} mr + \frac{1}{3} m r_0+ r_0^2 \dot{\varphi}_0 )  $$
provided $r > 0$. Hence by the first line of \eqref{particletrajectoryODEVM} we obtain
$$ \ddot{r} (s) \leq  \frac{1}{r^3} \max \{  ( ms-\frac{1}{3} mr + \frac{1}{3} m r_0 + r_0^2 \dot{\varphi}_0 )^2 , ( \frac{1}{3} mr - \frac{1}{3} m r_0 + r_0^2 \dot{\varphi}_0 )^2  \} + \frac{m}{r} + m \max\{ \dot{\varphi}, 0 \}  ,  $$
$$ \ddot{r} (s) \geq   \frac{1}{r^3}  \min \{  ( -\frac{1}{3} mr + \frac{1}{3} m r_0 + r_0^2 \dot{\varphi}_0 )^2 , ( -ms+ \frac{1}{3} mr - \frac{1}{3} m r_0 + r_0^2 \dot{\varphi}_0 )^2  \}   - \frac{m}{r} - m | \dot{\varphi} |  .  $$
Since $ - \frac{1}{2}  mr_0 + r_0^2  \dot{\varphi}_0 > 0$, and for $s \in [0, T_0]$, $s \leq T_0 \leq \frac{1}{100} r_0$, we deduce from the inequalities above that
$$ 0 < r^{-2} (   - \frac{1}{2} m r_0+ r_0^2 \dot{\varphi}_0 )   \leq \dot{\varphi} \leq r^{-2} (   \frac{1}{2} m r_0+ r_0^2 \dot{\varphi}_0 ) ,  $$
$$ \ddot{r} (s) \leq  r^{-3} \max \{  (  -\frac{1}{2} mr + \frac{1}{2} m r_0 + r_0^2 \dot{\varphi}_0 )^2 , ( \frac{1}{2} mr - \frac{1}{2} m r_0 + r_0^2 \dot{\varphi}_0 )^2  \} + mr^{-1} + m \max\{ \dot{\varphi}, 0 \}  ,  $$
$$ \ddot{r} (s) \geq  r^{-3}  \min \{  ( -\frac{1}{2} mr + \frac{1}{2} m r_0 + r_0^2 \dot{\varphi}_0 )^2 , (  \frac{1}{2} mr - \frac{1}{2} m r_0 + r_0^2 \dot{\varphi}_0 )^2  \}   - mr^{-1} - m | \dot{\varphi} |  .  $$
Let us denote $ \mathcal{C} = \frac{1}{2} m r_0 + r_0^2 \dot{\varphi}_0 $. Using $r(s) \leq r_0$ for $s \in [0, T_0]$, from the inequality $ \ddot{r} \leq r^{-3} \max \{  ( -\frac{1}{2} mr + \mathcal{C} )^2 , ( \frac{1}{2} mr - \frac{1}{2} m r_0 + r_0^2 \dot{\varphi}_0 )^2  \} + mr^{-1} + m \max\{ \dot{\varphi}, 0 \}  $ and $ 0 < r^{-2} (   - \frac{1}{2} m r_0+ r_0^2 \dot{\varphi}_0 )   \leq \dot{\varphi} \leq r^{-2} \mathcal{C}  $, we have
\begin{equation*}
\begin{split}
\ddot{r} 
& \leq r^{-3}  ( - \frac{1}{2} m r +  \mathcal{C} )^2  + mr^{-1} + m  r^{-2} \mathcal{C}   \\
& \leq r^{-3}  \mathcal{C}^2  + mr^{-1} + m  r^{-2} \mathcal{C}  .  \\
\end{split}
\end{equation*}
Multiplying this inequality by $\dot{r}$, we obtain
$$ \frac{1}{2} \big( \frac{d (\dot{r}^2)}{ds} \big)  \geq r^{-3}  \dot{r}  \mathcal{C}^2  +  m r^{-1} \dot{r} +  m  r^{-2}  \dot{r} \mathcal{C} .  $$
Integrating yields
\begin{equation} \label{E:behaviorofcharacteristicsVM-eq4}
\dot{r}^2 \geq   \dot{r}_0^2 +  (r_0^{-2} -r^{-2} ) \mathcal{C}^2  + 2m  \ln (\frac{r}{r_0}) + 2 m ( r_0^{-1}  - r^{-1}  ) \mathcal{C}  . 
\end{equation}

For $s \in [0, T_0]$, multiplying the inequality \eqref{E:behaviorofcharacteristicsVM-eq4} by $r^2$ and using $ r (s) \leq r_0 $, we obtain
\begin{equation}
\begin{split}
\big( \frac{1}{2} \frac{d}{ds} (r^2) \big)^2 
& \geq \dot{r}_0^2 r^2 +  (r_0^{-2} r^2 - 1 ) \mathcal{C}^2  + 2m  \ln (\frac{r}{r_0}) r^2 + 2 m ( r_0^{-1}  r^2 - r  ) \mathcal{C}  \\
& = [ \dot{r}_0^2 + r_0^{-2} \mathcal{C}^2 + 2 m r_0^{-1} \mathcal{C}  ] r^2   + 2m  \ln (\frac{r}{r_0}) r^2  - 2 m r \mathcal{C} - \mathcal{C}^2 . \\
\end{split}
\end{equation}
By the property of the function $y= x^2 \ln x$, we have $\ln (\frac{r}{r_0}) r^2 = \ln (r) r^2 - \ln (r_0) r^2 \geq - \frac{1}{2e} -  r^2 \ln r_0   \geq -1-  r^2\ln r_0 $. Hence we have
\begin{equation}
\begin{split}
\big( \frac{1}{2} \frac{d}{ds} (r^2) \big)^2 
& \geq [ \dot{r}_0^2 + r_0^{-2} \mathcal{C}^2  + 2 m r_0^{-1} \mathcal{C}  ] r^2 + 2m (-1-  r^2\ln r_0)    - 2 m r_0 \mathcal{C} - \mathcal{C}^2 \\
& = [ \dot{r}_0^2 + r_0^{-2} \mathcal{C}^2  + 2 m r_0^{-1} \mathcal{C} -2m 
\ln r_0 ] r^2   - 2 m r_0 \mathcal{C} - \mathcal{C}^2 -2m \\
& = \mathcal{A} r^2 - \mathcal{B} , \\
\end{split}
\end{equation}
where 
\begin{equation}
\begin{split}
\mathcal{A}
&  :=   \dot{r}_0^2 + r_0^{-2} \mathcal{C}^2   + 2 m r_0^{-1} \mathcal{C} - 2m \ln r_0   \\
& =   \dot{r}_0^2 + r_0^{-2} (  \frac{1}{2} m r_0 + r_0^2 \dot{\varphi}_0 )^2   + 2 m r_0^{-1} (   \frac{1}{2} m r_0+ r_0^2 \dot{\varphi}_0 ) - 2m \ln r_0 ,   \\
\end{split}
\end{equation}
\begin{equation}
\begin{split}
\mathcal{B}
&  :=   2 m r_0 \mathcal{C} + \mathcal{C}^2 + 2m   \\
& =   2 m r_0 (   \frac{1}{2} m r_0+ r_0^2 \dot{\varphi}_0 ) + (  \frac{1}{2} m r_0 + r_0^2 \dot{\varphi}_0 )^2 + 2m .   \\
\end{split}
\end{equation}
With the assumptions we have $\mathcal{A} > 0 $ and $\mathcal{B} > 0$.




If $r(s) > \sqrt{\frac{\mathcal{B}}{\mathcal{A}}}$, then $r_0 > \sqrt{\frac{\mathcal{B}}{\mathcal{A}}}$. Moreover, $\mathcal{A} r_0^2 - \mathcal{B}  > 0$ due to the assumptions. We have 
$$- \frac{1}{2} \frac{d}{ds} (r(s)^2) \big/ \sqrt{\mathcal{A} r(s)^2 - \mathcal{B}} \geq 1 \ , $$
which implies
$$  \frac{d}{ds} \sqrt{\mathcal{A} r(s)^2 - \mathcal{B}} \leq - \mathcal{A} \ . $$
Integrating gives
$$ \sqrt{\mathcal{A} r(s)^2 - \mathcal{B}} \leq \sqrt{\mathcal{A} r(0)^2 - \mathcal{B}} - \mathcal{A}s \ . $$
Therefore
\begin{equation}
\begin{split}
r(s)^2 
& \leq \mathcal{A} s^2 + r_0^2 - 2s \sqrt{\mathcal{A} r_0^2 - \mathcal{B}} \\
& = (r_0 - \sqrt{\mathcal{A}  - \mathcal{B} r_0^{-2}} s )^2 + \mathcal{B} r_0^{-2} s^2 \ . \\
\end{split}
\end{equation}
Hence in general 
$$ r(s)^2 \leq \max \{ \frac{\mathcal{B}}{\mathcal{A}}, (r_0 - \sqrt{\mathcal{A}  - \mathcal{B} r_0^{-2}} s )^2 + \mathcal{B} r_0^{-2} s^2 \} \ . $$
Noticing that $ \frac{\mathcal{B}}{\mathcal{A}}$ is actually the minimum of the parabola $(r_0 - \sqrt{\mathcal{A}  - \mathcal{B} r_0^{-2}} s )^2 + \mathcal{B} r_0^{-2} s^2 $ (with respect to $s$), attained at $s = s_m : = \frac{\sqrt{\mathcal{A} r_0^2 -\mathcal{B}}}{\mathcal{A}}$. Therefore we conclude, for all $s \in [0, T_0]$,
\begin{equation}
\begin{split}
r(s)^2 
&  \leq (r_0 - \sqrt{\mathcal{A}  - \mathcal{B} r_0^{-2}} s )^2 + \mathcal{B} r_0^{-2} s^2 .  \\
\end{split}
\end{equation}

This completes the proof of the lemma.

\end{proof}

\section{Proof of Theorem \ref{mainresultVM}} 
\label{S:mainthmproof}

Recall that 
\begin{equation}
S(t) := \{  (r,\varphi,  \dot{r} , \dot{\varphi} ) : f(t, r,\varphi,   \dot{r} , \dot{\varphi}  ) > 0  \}  ,
\end{equation}
\begin{equation}
S(0) := \{  (r,\varphi,  \dot{r} , \dot{\varphi} ) : f_0(r,\varphi,  \dot{r} , \dot{\varphi}   ) > 0  \}  . 
\end{equation}
 


We now give the proof of Theorem \ref{mainresultVM}.

\begin{proof}

We only prove the case $\beta=1$ here since the case $\beta > 1$ is proved in a similar way. 

Without loss of generality, we assume $0 <\eta < 1$ and $N >1$. 
Let $\epsilon \in (0,1)$ be a small positive constant to be determined. Let $k$, $l $ and $\alpha$ be positive constants satisfying $k< \frac{1}{3} l$, $\alpha < l-k$, $\alpha > 4k$, $l > 10 \alpha + 10k$. 



Let $H : [0, +\infty) \times [0, +\infty) \rightarrow [0, +\infty)$ be any function satisfying $ \int_{\mathbb{R}^2} H(|u_1|^2, |u_2|^2 ) du_1 du_2 = 1 $ with $\supp (H) \subset [0, 1) \times (0, 1)$, and rescale it for any $\epsilon \in (0, 1)$:
$$ H_\epsilon (|u_1|^2, |u_2|^2) = \frac{1}{\epsilon^{4k}}   H \big( \frac{|u_1|^2}{\epsilon^{4k}}, \frac{|u_2|^2}{\epsilon^{4k}} \big)   \ . $$
so that $ \int_{\mathbb{R}^2} H_\epsilon(|u_1|^2, |u_2|^2) du_1 du_2 = 1$ and $\supp (H_\epsilon) \subset [0, \epsilon^{2k}) \times (0, \epsilon^{2k})$. For any $\epsilon >0$, $x$, $v \in \mathbb{R}^2$, define 
$$ h_{\epsilon} (x, v) =  H_\epsilon \big(  | r- |\dot{r}| \epsilon^{2l-k} |^2 , |\dot{\varphi}- \epsilon^{-l}|^2 \big) \ . $$
We choose the cut-off function $\chi_{0, 1} \in C^\infty ((0, \infty); [0, 1] ) $ satisfying $|\chi_{0, 1} (x)| \leq 1$, $\chi_{0, 1} (x) = 1  $ for $|x| <  \frac{1}{2}  $ or, and $\chi_{0, 1} (x) =0$ for $|x| > 1$. Let $\chi_{m, n} (x) : = \chi_{0, 1} (\frac{1}{n} (x-m))$, and take $\chi (r) = \chi_{\frac{3}{4}, \frac{1}{4} } (r)$, then $\chi (r)$ is a smooth function supported on $[\frac{1}{2}  , 1 ]$, and $\chi (r ) = 1 $ for $r \in [\frac{5}{8} , \frac{7}{8}  ]$. 
For $r>0$, we define 
\begin{equation} \label{chiVPM}
\| \frac{d^k}{dr^k} ( \chi (r)  ) \|_{L^\infty} \leq c_k   . 
\end{equation}
In particular $c_0 =1$. Denote $d_k : = \sum_{j=0}^k c_j$.  

Let $\epsilon \in (0, 1)$ be a small positive constant to be determined later.
We choose the initial data on the density distributions of the ions and the electrons to be 
\begin{equation} \label{f0VPM}
f_0 (x, v) = \epsilon^{\alpha-k+2l} h_{\epsilon} (x, v)  \chi (|x|) \Gamma (\dot{r})  \ , 
\end{equation}
where $\Gamma (\dot{r}) = 0$ for $\dot{r} \geq 0$, and $\Gamma (\dot{r}) =1$ for $\dot{r} <0$. Notice that $\Gamma$ does not have any effect on the smoothness of $f_0 (x, v)$. Note that $\supp  f_0  $ is a region near the set $\{ r- |\dot{r}| \epsilon^{2l-k}=0, \dot{\varphi}=  \epsilon^{-l}  \}$. Moreover, the support of $f_0$ with respect to $x$ is $\{ \frac{1}{2} \leq |x| \leq 1 \}$. Let $M= \iint f (t, x, v) dv dx = \iint f_0 (x, v) dv dx = 2 \pi \int_0^{+\infty} \int_{\mathbb{R}^2} r f_0 (r, v) dv dr $. Using $dv = dv_1 dv_2 = r d\dot{r} d \dot{\varphi}$, we compute,   
\begin{equation}
\frac{\pi}{16} \epsilon^\alpha  \leq M  \leq 2 \pi \epsilon^\alpha . 
\end{equation}
For any $(r_0,\varphi_0,  \dot{r}_0 , \dot{\varphi}_0 )  \in S(0)$, we select $\epsilon$ small enough, so the following holds:
\begin{equation}
\frac{1}{2} \leq r_0 \leq 1, \ | r_0 - |\dot{r}_0| \epsilon^{2l-k} |^2  < \epsilon^{4k} , \ |\dot{\varphi}_0- \epsilon^{-l}|^2 < \epsilon^{4k}   . 
\end{equation}  
Hence we have
\begin{equation} \label{VPM-eq1}
\begin{split}
& \frac{1}{2} \leq r_0 \leq 1, \\
& \frac{1}{2} \epsilon^{k-2l} <  (r_0 - \epsilon^{2k}) \epsilon^{k-2l} < |\dot{r}_0| < (r_0 + \epsilon^{2k}) \epsilon^{k-2l} < 2 \epsilon^{k-2l}, \\ 
& \frac{1}{2} \epsilon^{-l} <  \epsilon^{-l} - \epsilon^{2k} < \dot{\varphi}_0 < \epsilon^{-l} + \epsilon^{2k} < 2 \epsilon^{-l} , \\
\end{split}
\end{equation}  
with $\dot{r}_0 < 0$. 

From \eqref{f0VPM}, we have
\begin{equation}
\rho (0, r) =  \frac{1}{2} \epsilon^\alpha  \chi (r)  . 
\end{equation}
We let $\epsilon$ be small enough such that
\begin{equation} \label{VPM-eq2}
 \| \rho (0) \|_{L^\infty} \leq 2 \pi \epsilon^\alpha  \leq  \frac{1}{10} \eta , 
\end{equation} 
Noticing that $\rho (t, 0)$ is only a function of $r$, we choose $\epsilon$ to be small enough such that
\begin{equation} \label{VPM-eq6}
\| \rho (0) \|_{C^1_x}  \leq   \| \frac{\partial }{\partial r} \rho (0) \|_{L^\infty} \leq \frac{1}{2} \epsilon^\alpha   \| \frac{\partial }{\partial r }  (  \chi (r) ) \|_{L^\infty} \leq \frac{1}{2} \epsilon^\alpha d_1  \leq \eta  . 
\end{equation}
For the electric field, we let $\epsilon$ be small enough such that for all $r \in [1/2, +\infty)$, 
\begin{equation}
\begin{split}
| E_r (0, r) | 
& = \frac{1}{r}  \int_0^{r}  \int_{\mathbb{R}^2} s f(0, s, v) ds dv    \\
& \leq \frac{1}{r}  \int_0^{+\infty}  \int_{\mathbb{R}^2} s f(0, s, v) ds dv \\
& \leq  2 \cdot \frac{\pi}{16} \epsilon^\alpha \cdot \frac{1}{2\pi} \\
& \leq \eta ,  \\
\end{split}
\end{equation}
and moreover, 
\begin{equation}
\begin{split}
| \partial_r E_r (0, r) | 
& = \big| \partial_r  \big( \frac{1}{r}  \int_0^{r}  \int_{\mathbb{R}^2} s f(0, s, v) ds dv  \big) \big|  \\
& \leq  | \frac{1}{r^2}  \int_0^{r}  \int_{\mathbb{R}^2} s f(0, s, v) ds dv  | +  |  \frac{1}{r}   \int_{\mathbb{R}^2} r f(0, r, v)  dv  |  \\
& \leq  | \frac{1}{r^2}  \int_0^{+\infty}  \int_{\mathbb{R}^2} s f(0, s, v) ds dv  | +  |     \int_{\mathbb{R}^2}  f(0, r, v)  dv  |  \\
& \leq  4 \cdot \frac{\pi}{16} \epsilon^\alpha \cdot \frac{1}{2\pi} + \| \rho (0) \|_{L^\infty} \\
& \leq \eta .  \\
\end{split}
\end{equation}
Notice that for all $r \in [0, 1/2)$ we have $| E_r (0, r) |  =0$, $| \partial_r E_r (0, r) |  =0$. 
Furthermore, we set the initial data (with $\epsilon >0 $ sufficiently small so that $\epsilon^{\alpha} < \eta$) such that 
\begin{equation}
| (E_\varphi , B) (0, r) | \leq   \frac{\epsilon^{\alpha}}{200 r }  \leq \frac{12 \epsilon^{\alpha-4k-l}}{r} \quad   \text{and} \quad  \| (E_\varphi , B) (0) \|_{C^1} \leq   \frac{\epsilon^{\alpha}}{200 } . 
\end{equation}
Also, we have
\begin{equation}   
| E_r (t, r) |  \leq \frac{M}{2 \pi r}  \leq \frac{\epsilon^\alpha}{r}  
\end{equation}
for all $t > 0$. Now, let $\epsilon$ be small enough such that the following conditions hold with $m = 100  \epsilon^{\alpha-4k-l}$. These conditions correspond to the assumptions in Lemma \ref{L:behaviorofcharacteristicsVPM}:
\begin{equation} \label{VPM-eq5}
\begin{split}  
& | E_r (t, r) |  \leq \frac{\epsilon^\alpha}{r}  \leq \frac{m}{3r}  ,   \\
& - \frac{1}{2}  mr_0 + r_0^2  \dot{\varphi}_0 > 0 ,  \\
& \mathcal{A}  := \dot{r}_0^2 + r_0^{-2} (  \frac{1}{2} m r_0 + r_0^2 \dot{\varphi}_0 )^2  + 2 m r_0^{-1} (   \frac{1}{2} m r_0+ r_0^2 \dot{\varphi}_0 ) - 2m \ln r_0 \\
& \quad  \in  [  r_0^2 ( \epsilon^{2k-4l} - 100 \epsilon^{-2l} ) ,  r_0^2  ( \epsilon^{2k-4l} +   100 \epsilon^{-2l}  ) ]   \subset [ \frac{1}{2} r_0^2 \epsilon^{2k-4l} , 2 r_0^2 \epsilon^{2k-4l} ]    ,  \\
& \mathcal{B} :=  2 m r_0 (   \frac{1}{2} m r_0+ r_0^2 \dot{\varphi}_0 ) + (  \frac{1}{2} m r_0 + r_0^2 \dot{\varphi}_0 )^2  + 2m \\
 & \quad \in  [ r_0^4 ( \epsilon^{-2l} - 100 \epsilon^{\alpha-4k-2l}  ) , r_0^4 ( \epsilon^{-2l} +   100 \epsilon^{\alpha-4k-2l} )  ]   \subset  [ \frac{1}{2} r_0^4 \epsilon^{-2l} , 2 r_0^4 \epsilon^{-2l} ]   ,  \\
& \mathcal{A} r_0^2 - \mathcal{B}  > 0 . \\ 
\end{split}
\end{equation}
Let $T := \epsilon^{-k+2l} - 300 \epsilon^{\alpha-5k+3l} - 300 \epsilon^{k+2l} - 300 \epsilon^{3l-2k}    < \epsilon^{-k+2l}  \ll  \frac{1}{100} r_0$.





\begin{flushleft}
\textbf{Claim.} For $t \in [0, T]$, $r \leq 10$, we have
\begin{equation}  \label{E:claim-1}
\begin{split}
| ( E_\varphi, B) (t, r) | 
& \leq \min\{ \frac{\epsilon^{\alpha}}{200r} ,  \| (E_\varphi , B) (0) \|_{L^\infty} \} + \frac{12\epsilon^{\alpha-4k-l}}{r}   \\
& \leq \min\{ \frac{\epsilon^{\alpha}}{200r} , \frac{\epsilon^{\alpha}}{200} \} +  \frac{12\epsilon^{\alpha-4k-l}}{r}   \leq \frac{24\epsilon^{\alpha-4k-l}}{r}   \leq \frac{m}{3r}  .  \\
\end{split}
\end{equation}
and 
\begin{equation}   \label{E:claim-2}
\inf_{ (r_0 ,\varphi_0,  \dot{r}_0 , \dot{\varphi}_0 )\in S(0)} r (t,  r_0 , \varphi_0 ,  \dot{r}_0 , \dot{\varphi}_0 )  \geq 10 t ,
\end{equation}
\begin{equation}   \label{E:claim-3}
\begin{split}
& \dot{r} < 0 . \\
\end{split}
\end{equation}


\vskip 0.4cm

\textit{Proof of Claim.} We carry out a bootstrapping argument. If \eqref{E:claim-1}, \eqref{E:claim-2} and \eqref{E:claim-3} hold on $[0, \tilde{T}]$ with $\tilde{T} < T$, we can show that on $[0, \tilde{T}]$, \eqref{E:claim-1} holds with $\frac{24\epsilon^{\alpha-4k-l}}{r} $ replaced by $\frac{20\epsilon^{\alpha-4k-l}}{r} $ on the right hand side using the argument in Lemma \ref{L:est-E-B} (that is, $| ( E_\varphi, B) (t, r) |  \leq \frac{20\epsilon^{\alpha-4k-l}}{r} $). Indeed, for $t \in [0, \tilde{T}]$, we have
\begin{equation}
\begin{split}
\rho (t, r) 
& = \int_{\supp f } f(s, r, \dot{r}, \dot{\varphi})  r d \dot{r} d \dot{\varphi} \leq  \int_{\supp f } \| f_0 \|_{L^\infty}  r d \dot{r} d \dot{\varphi}  \\
& \leq  \epsilon^{\alpha-k+2l} \epsilon^{-4k} \epsilon^{k-2l}  \epsilon^{-l} \epsilon^{k-l}   = \epsilon^{\alpha-3k-2l}   \\
\end{split}
\end{equation}
and
\begin{equation}  
\begin{split}
& \dot{r}(t) < 0 , \quad  \sup_{ (r_0 ,\varphi_0,  \dot{r}_0 , \dot{\varphi}_0 )\in S(0)}  |\dot{\varphi}_0|  \leq  2 \epsilon^{-l}  ,  \quad  \sup_{ (r_0 ,\varphi_0,  \dot{r}_0 , \dot{\varphi}_0 )\in S(0)}  |r_0|  \leq  1  .  \\ 
\end{split}
\end{equation}
Since $\dot{r}(t) < 0$ on $t \in [0, \tilde{T}]$, we have $r \leq 1 < 10$ on $t \in [0, \tilde{T}]$ for all particle trajectories. Integrating the second line of \eqref{particletrajectoryODEVM} over $[0, t]$ and using $ |(E_r, E_\varphi, B) (t)| \leq \frac{m}{3r} $, we learn that for all $t \in [0, \tilde{T}]$,
$$  -mt+ \frac{1}{3} mr(t) - \frac{1}{3} m r_0+ r_0^2 \dot{\varphi}_0   \leq r(t)^2 \dot{\varphi} (t) \leq mt -\frac{1}{3} mr(t) + \frac{1}{3} m r_0+ r_0^2 \dot{\varphi}_0  ,  $$
from which we deduce $ R^2 \dot{\varphi}  \leq \frac{4}{3} r_0^2 \dot{\varphi}_0 $ using the smallness of $\epsilon$. Repeat the proof of Lemma \ref{L:est-E-B}, we can show that the result in Lemma \ref{L:est-E-B} holds on $[0, \tilde{T}]$ with $\frac{6KK_1 T}{R}$ replaced by $\frac{4KK_1 T}{R}$ \footnote{In the proof of Lemma \ref{L:est-E-B}, replace $6KK_1 T$ by $4KK_1 T$ in the hypothesis "for some $t \in [0, T]$ $\sup_{r >0} |rB(t, r)| \geq r \| ( E_{\varphi, 0}, B_0) \|_{L^\infty}  + 6KK_1T $",  in the definition $t_0 := \inf \{ t>0 : \sup_{r >0} |rB(t, r)| \geq  r \| ( E_{\varphi, 0}, B_0) \|_{L^\infty}  + 6KK_1 T \}$, and in the left hand side of \eqref{E:est-E-B-eq5}. Then continue with the proof to show an adapted version of Lemma \ref{L:est-E-B} with $\frac{6KK_1 T}{R}$ replaced by $\frac{4KK_1 T}{R}$ in \eqref{E:est-E-B-eq4}.}, we have 
\begin{equation}   
\begin{split}
| ( E_\varphi, B) (t, r) | 
&  \leq  \| (E_\varphi , B) (0) \|_{L^\infty}  + \frac{4 \epsilon^{\alpha-3k-2l} \cdot 2 \epsilon^{-l}  \cdot \epsilon^{-k+2l} }{r}  \\
& \leq \frac{\epsilon^{\alpha}}{200}  +  \frac{8 \epsilon^{\alpha-4k-l}}{r} \leq   \frac{20 \epsilon^{\alpha-4k-l}}{r}  \\ 
\end{split}
\end{equation}
for $t \in [0, \tilde{T}]$, $r \leq 10$.

Now, with \eqref{VPM-eq5}, we apply the ODE analysis in Lemma \ref{L:behaviorofcharacteristicsVPM}. We deduce that for all $(r_0 ,\varphi_0,  \dot{r}_0 , \dot{\varphi}_0 )\in S(0)$ and $t \in [0, \tilde{T}]$, 
$$ r(t)^2    \leq (r_0 - \sqrt{\mathcal{A}  - \mathcal{B} r_0^{-2}} t )^2 + \mathcal{B} r_0^{-2} t^2 .   $$
We conclude that on $[0, \tilde{T}]$ \eqref{E:claim-2} and \eqref{E:claim-3} hold with $10 t$ replaced by $100 t$ in \eqref{E:claim-2} and $\dot{r} < 0$ on $[0, \tilde{T}]$ (provided that $\epsilon$ is sufficiently small). Therefore the interval $[0, \tilde{T}]$ on which \eqref{E:claim-1}, \eqref{E:claim-2} and \eqref{E:claim-3} hold can be prolonged.


\end{flushleft}


By Lemma \ref{L:behaviorofcharacteristicsVPM}, we have the following estimate from the argument above for sufficiently small $\epsilon$:
\begin{equation}
\begin{split}
r(T)^2  
&  \leq (r_0 - \sqrt{\mathcal{A}  - \mathcal{B} r_0^{-2}} T )^2 + \mathcal{B} r_0^{-2} T^2  \\
& \leq  [ r_0 - \sqrt{ r_0^2  ( \epsilon^{2k-4l} -  100 \epsilon^{-2l}  ) - r_0^4 ( \epsilon^{-2l} + 100 \epsilon^{\alpha-4k-2l}  )  r_0^{-2}}  \\
& \quad \cdot (\epsilon^{-k+2l} - 300 \epsilon^{\alpha-5k+3l} - 300 \epsilon^{k+2l} - 300 \epsilon^{3l-2k} ) ]^2  + 2 r_0^2 \epsilon^{-2l-2k+4l}    \\
& \leq  10000 r_0^2 \epsilon^{2l-2k}  + 100000 r_0^2 \epsilon^{2\alpha-8k+2l} + 2 r_0^2 \epsilon^{2l-2k}   \\
& \leq  40000 r_0^2 \epsilon^{2l-2k} .  \\
\end{split}
\end{equation}
Therefore 
\begin{equation}
r(T) \leq 200 r_0^2 \epsilon^{l-k} . 
\end{equation}

Now we have
\begin{equation}
\begin{split}
R_+ = R_+ (T)
&  := \sup_{ (r_0 ,\varphi_0,  \dot{r}_0 , \dot{\varphi}_0 )\in S(0)} r (T,  r_0 , \varphi_0 ,  \dot{r}_0 , \dot{\varphi}_0 )  \in [ \frac{1}{2} r_0 \epsilon^{l-k} , 200 r_0 \epsilon^{l-k} ]  \\
&  \subset  [ \frac{1}{4}  \epsilon^{l-k} , 200  \epsilon^{l-k} ] .   \\
\end{split}
\end{equation}
By the conservation of mass, we have
\begin{equation}
\| \rho (T) \|_{L^\infty_{|x| \leq R_+}} \geq  \frac{M}{\pi R_+^2} \geq \frac{\pi}{16} \frac{ \epsilon^\alpha}{ \pi  (200 \epsilon^{l-k})^2} = \frac{1}{640000} \epsilon^{\alpha-2l+2k}  . 
\end{equation}
By \eqref{E:rep-E-B-eq4}, 
\begin{equation}
\begin{split}
| E_r (T, R_+) | 
& = \frac{1}{R_+}  \int_0^{R_+}  \int_{\mathbb{R}^2} s f(T, s, v) ds dv    \\
& = \frac{1}{R_+}  \int_0^{+\infty}  \int_{\mathbb{R}^2} s f(T, s, v) ds dv \\
& \geq  \frac{1}{200 \epsilon^{l-k} }  \frac{\pi}{16}  \epsilon^\alpha \frac{1}{2 \pi}  \\
& =  \frac{1}{6400} \epsilon^{\alpha-l+k} , \\
\end{split}
\end{equation}
Selecting $\epsilon$ small enough, we have
\begin{equation}
\| \rho (T) \|_{L^\infty_{|x| \leq R_+}} \geq   \frac{1}{640000} \epsilon^{\alpha-2l+2k} > N  ,  \quad 
\| E_r (T) \|_{L^\infty_{|x| \leq R_+}} \geq  \frac{1}{6400} \epsilon^{\alpha-l+k} > N 
\end{equation}
with $R_+ < \epsilon_0$.

This completes the proof of Theorem \ref{mainresultVM}.

\end{proof}

\end{document}